\documentclass[10pt,oneside,a4paper]{article}
\usepackage{amsmath,amsthm}
\usepackage{hyperref}
\usepackage{geometry}
\usepackage{url}

\newtheoremstyle{mytheorem}
  {3pt}
  {3pt}
  {\itshape}
  {}
  {\bfseries}
  {.}
  {1em}
  {}
\theoremstyle{mytheorem}
\newtheorem*{theorem}{Theorem}

\newtheorem*{thm}{Theorem (Rosenberg-Stolz)}
\newtheorem*{thm 2}{Iso-enlargeable [$\asymp \frac{8\pi^2}{R^2}$]-Decay  Theorem}

\usepackage{sectsty}
\sectionfont{\centering}
\usepackage{secdot}

\theoremstyle{remark}
\newtheorem{remark}{Remark}

\begin{document}
\title{A note about scalar curvature on the total space of a vector bundle}
\author{Jialong Deng}
\date{}
\newcommand{\Addresses}{{
  \bigskip
  \footnotesize

  \textsc{Mathematisches Institut, Georg-August-Universität, Göttingen, Germany}\par\nopagebreak
  \textit{E-mail address}: \texttt{jialong.deng@mathematik.uni-goettingen.de}}}
\maketitle
\begin{abstract}
 We construct complete Riemannian metrics to show that the total space of  tangent bundles of orientable closed surfaces (except torus) admits complete uniformly PSC-metrics. It gives a partial  positive answer to one of Gromov's question. 
\end{abstract}

 The scalar curvature of a Riemannian metric on closed manifolds has been deeply studied (see $ \cite{MR1268192}$, $\cite{MR2408269}$,$\cite{MR3728662}$ and  references therein) for the reason of geometry and general relativity. In this paper, we consider the scalar curvature of a Riemannian metric on the total space of tangent bundle of the orietnable closed surfaces. It is a classic result that any noncompact Riemannian surface admits complete Riemannian metrics with constant negative scalar curvature. Bland and Kalka $\cite{MR987159}$ generalized it to manifolds of any dimension: any noncompact manifold of dimensions at least three admits complete Riemannian metrics with constant negative scalar curvature. In contrast to the negative scalar curvature, it exists obstructions to admit a complete Riemannian metric with positive scalar curvature (PSC-metric) on open manifold. If one does not require completeness, then the curvature will not restrict the topology of open manifold according to Gromov's Theorem \cite{MR0263103}. Gromov and Lawson [$\cite{MR720933}$, Corollary 6.13] proved that $X\times \mathbf{R}$ carries no complete PSC-metric. Here $X$ is a  closed  enlargeable manifold [$\cite{MR720933}$, Definition 5.5]. For instance, closed Riemannian manifolds with nonpositive sectional curvature are enlargeable manifolds. There is a widely open conjecture: for closed manifold $M^{n}$, $M^{n}\times \mathbf{R}$ admits a complete PSC-metric if and only if $M^{n}$ admits PSC-metrics. Gromov and Lawson also showed in $\cite{MR720933}$ that hyperbolic manifold times $ \mathbf{R}^{2}$ does not admit complete uniformly PSC-metric.  However, if the fiber is the vector space $\mathbf{R}^n$($n\geq 3$), the total space of vector bundles (trivial or nontrivial) will admit complete uniformly PSC-metric, as $\mathbf{R}^n$ admits $O(n)$-invariant complete uniformly PSC-metrics.\par
  Rosenberg and Stolz [$\cite{MR1268192}$, proposition 7.2] described a proof that $M^{n}\times \mathbf{R}^{2}$ admits complete PSC-metrics, where $M^{n} $ is a closed manifold. Moreover, Gromov proved the following theorem in [\cite{MR3816521}, section 4]:

\begin{thm 2}

Let a manifold $X$ admit a proper map of non-zero degree to  $Y \times \mathbf{R}^{2}$, where $Y$ is a compact iso-enlargeable manifold, then the scalar curvature of all complete Riemannian metric $g$ in $X$ restricted to concentric balls $B(R)=B_{x_0}(R)\subset X$ satisfy
   $\min\limits_{B(R)} Sc(X)\leq\frac{8\pi^2}{(R-R_0)^2}$ for some $R_0=R_0(X,g,x_0)$ and all $R\geq R_0$.
\end{thm 2}
Riemannian manifolds with non-positive sectional curvature are examples of iso-enlargeable manifolds and the definition of the iso-enlargeable manifold can be found in [\cite{MR3816521}, section 4, p.658]. Gromov posted a question in [\cite{MR3816521}, section 6]:\par
 \textit{ Do all products manifold $Y\times \mathbf{R}^{2}$, and, more generally, the total spaces of all $\mathbf{R}^{2}$-bundles admit complete metric $g$ with $Sc(g)\geq 0$, i.e. scalar curvature $\geq 0$? Do, for example, such metrics $g$ exist for compact manifold $Y$ which admit metrics with strictly negative sectional curvature?}\par
   Gromov also pointed out in the paper that the best candidates of this kind of manifolds with no complete PSC-metric on them are non-trivial $\mathbf{R}^{2}$-bundles over surfaces of genus at least 2. However, it may has a positive answer.\par

\begin{thm}
If $Y$ is a smooth manifold without boundary (compact or noncompact), then $Y \times \mathbf{R}^{2}$ admits complete PSC-metrics.

\end{thm}

\begin{remark}
Though Rosenberg and Stolz [\cite{MR1268192}, proposition 7.2, p.263] only deal with the case of closed manifold times  $\mathbf{R}^{2}$, their argument can be used to here. Since the open manifold can be endowed with a complete Riemannian metric with constant negative scalar curvature by the theorem of  Bland and Kalka $\cite{MR987159}$.  Thus the theorem is attributed to Rosenberg and Stolz. 
\end{remark}

\begin{theorem}
The total space of tangent bundle of orientable closed surface (except torus) admits complete uniformly PSC-metrics.
\end{theorem}

\begin{proof}
The total space of tangent bundle of $\mathbf {S}^{2}$ admits complete metrics of positive Ricci curvature and  the Cheeger-Gromoll metric on it is a complete uniformly PSC-metric [$\cite{MR1908215}$, Theorem 1.1, property 3.3]. The trivialization of the tangent bundle $\mathbf{T}^{2}$ makes it carry no complete uniformly PSC-metric as $\mathbf{T}^{2}$ is enlargeable [\cite{MR720933}, Theorem 7.3]. But it admits complete PSC-metrics.\par
    Let $\Sigma$ be the orientable surface with genus at least 2 which is endowed a hyperbolic metric $g$. Then the Levi-Civita connection $\nabla$ on $\Sigma$ induces a natural splitting of the tangent space $T_{(p,U)}T\Sigma$  into
\begin{align*}
T_{(p,u)}T\Sigma=V_{(p,U)}T\Sigma \oplus H_{(p,U)}T\Sigma
\end{align*}
Where $(p,u) \in T\Sigma $, $V_{(p,U)}T\Sigma$ is the vertical space, i.e. $Ker\pi_{\ast(p,U)},  \pi: T\Sigma \to \Sigma$ and $H_{(p,U)}T\Sigma$ is the horizontal space at $(p,U)$ obtained by using $\nabla$. A horizontal curve in $T\Sigma$ corresponds to a vector field on $\Sigma$ which is parallel with respect to connection $\nabla$ on $(\Sigma , g)$. A vector on $T\Sigma$ is horizontal if it is tangent to a horizontal curve and vertical if it is tangent to a fiber. There exists the horizontal (resp. vertical) distribution $HT\Sigma$  (resp.$VT\Sigma$) and the direct sum decomposition: $TT\Sigma = HT \Sigma \oplus VT\Sigma$. This rises to the horizontal and vertical lifts $X^h$, $X^v$ of a vector field $X$ on $\Sigma$. The metric $g^{a,b}$ on $T\Sigma$ be constructed by:
\[\begin{cases}
g_{(p,U)}^{a,b}(X^h,Y^h)=g_p(X,Y), \quad  g_{(p,U)}^{a,b}(X^h,Y^v)=0\\
g_{(p,U)}^{a,b}(X^v,Y^v)=0.01g_p(X,Y)+[1+\frac{1}{2}g_p(U,U)]g_p(X,U)g_p(Y,U)
\end{cases}\]

 The metric induced by $100\times g^{a,b}$ is no less than  the metric induced by Cheeger-Gromoll metric $\cite{MR0309010}$ which is complete. Thus, it implies the completeness of $g^{a,b}$.
Let $t=\frac{1}{2}g_p(U,U)$, then the metric $g^{a,b}$ is a special example of the general metric on $T\Sigma$: 
\[\begin{cases}
\widetilde{g_{(p,U)}^{a,b}}(X^h,Y^h)=g_p(X,Y),\quad \widetilde{g_{(p,U)}^{a,b}}(X^h,Y^v)=0\\
\widetilde{g_{(p,U)}^{a,b}}(X^v,Y^v)=a(t)g_p(X,Y)+b(t)g_p(X,U)g_p(Y,U)
\end{cases}\]
The general metric was firstly defined  in $\cite{MR1726175}$ and Munteanu computed its scalar curvature [$\cite{MR2406440}$, Corollary 2.16]: Let $(M_C^n,g)$ be a space form, then 
\begin{align*}
Sc(\widetilde{g} )&=(n-1)\left\{nC+t[2-3a(t)]C^2-\frac{nF_2+4tF_3}{a(t)}\right\}
\end{align*}
where n is the dimension of $M$, $C$ is the constant of sectional curvature of $(M^n, g)$, 
\begin{align*}
F_2=L-M(1+2tL),\quad F_3=N-(M'+M^2+2tMN)\\
L=\frac{a'(t)}{2a(t)}, \quad M=\frac{2b(t)-a'(t)}{2[a(t)+2tb(t)]}\quad and \quad N=\frac{a(t)b'(t)-2a'(t)b(t)}{2a(t)[a(t)+2tb(t)]}.
\end{align*}
In putting the data $n=2$, $C=-1$, $a=0.01$, $b=1+t$ to the formulas, one gets 
\begin{align*}
L=0, \quad M=\frac{1+t}{0.01+2t(1+t)},\quad N=\frac{1}{2[0.01+2t(1+t)]},\quad M'=\frac{0.01-2(1+t)^2}{[0.01+2t(1+t)]^2}\\
2tMN=\frac{2t(1+t)}{[0.01+2t(1+t)]^2}, \quad F_2=\frac{-(1+t)}{0.01+2t(1+t)}, \quad F_3=\frac{(1+t)^2-0.005}{[0.01+2t(1+t)]^2}
\end{align*}
Then the scalar curvature of $g^{a,b}$ at $(p,U)$ is
\begin{align*}
Sc(g^{a,b})(p,U)&=-2+1.97t+\frac{2+4t}{[0.01+2t(1+t)]^2}\\
&=\frac{7.88t^5+7.76t^4-8.0412t^3-8.0012t^2+3.920197t+1.9998}{[0.01+2t(1+t)]^2}  :=\frac{g(t)}{h(t)}.
\end{align*}

Since $g(t)>1$ for $t\geq 0$ and $g(t)$ increases faster than $h(t)$ when $t$ goes to positive infinity,  $Sc(g^{a,b}) (p,U)\geq C_1>0$ for all $(p,U)\in T\Sigma$. Thus, $g^{a,b}$ is a complete uniformly PSC-metric.
\end{proof}\par

 $\mathbf{Acknowledgment}$: I thank Alexander Engel for pointing out the error on my proof of Rosenberg-Stolz Theorem in the first draft, Thomas Schick and Chao Qian for many  stimulating conversations, Shuqiang Zhu for proof-read and the funding from China Scholarship Council.

\bibliographystyle{alpha}
\bibliography{reference}
\Addresses

\end{document}